\documentclass[11pt]{article}

\usepackage[english]{babel}
\usepackage{amsmath,amssymb,amsthm}
\usepackage{enumerate}
\usepackage[affil-it]{authblk}
\usepackage[pdftex]{graphicx,color}
\usepackage{color}
\usepackage{hyperref}
\usepackage{subcaption}
\usepackage{comment}
\usepackage{epstopdf}

\newtheorem{theorem}{Theorem}[section]

\newtheorem{lemma}[theorem]{Lemma}

\theoremstyle{definition}

\newtheorem*{remark}{Remark}

\newcommand{\Z}{\mathbb{Z}}

\newcommand{\R}{\mathbb{R}}

\newcommand{\mx}{\mathbf{x}}
\newcommand{\mX}{\mathbf{X}}

\newcommand{\mbU}{\mathbf{U}}

\newcommand{\E}{\operatorname{{\mathbb E}}}

\renewcommand{\P}{\operatorname{\mathbb{P}}}

\newcommand{\F}{\mathcal{F}}

\newcommand{\M}{\mathcal{M}}
\newcommand{\mcL}{\mathcal{L}}

\begin{document}

\title{Random Additions in Urns of Integers}

\author{Mackenzie Simper\footnote{\texttt{msimper@stanford.edu}}}
\affil{Department of Mathematics, Stanford University}
\date{}
\maketitle

\begin{abstract}
Consider an urn containing balls labeled with integer values. Define a discrete-time random process by drawing two balls, one at a time and with replacement, and  noting the labels. Add a new ball labeled with the sum of the two drawn labels. This model was introduced by Siegmund and Yakir ({\it Ann. Probab. 33(5), 2005}) for labels taking values in a finite group, in which case the distribution defined by the urn converges to the uniform distribution on the group. For the urn of integers, the main result of this paper is an exponential limit law. The mean of the exponential is a random variable with distribution depending on the starting configuration. This is a novel urn model which combines multi-drawing and an infinite type of balls. The proof of convergence uses the contraction method for recursive distributional equations. 
\end{abstract}

\section{Introduction}

In \cite{SiegmundYakir}, the following urn model was introduced: Balls in an urn are labeled with elements of a finite group $G$. To update the urn, two balls are drawn (independently and with replacement) and their labels recorded. A new ball with the product of the two labels is then added to the urn. When the initial configuration of the urn contains generators of the group, it is proven that the composition of the urn converges to uniform distribution on the group (an alternate proof is given in \cite{AbramsLandau}). A natural extension of the model is to allow the labels with values from infinite group. Specifically, this paper considers the same dynamics for an urn containing integers with the addition operation, and call such a process a \emph{$\Z$-urn}. 

The behavior of the urn should depend on the initial configuration of balls. Perhaps a natural first question to explore is the behavior of the urn initially started with an additive basis for $\Z$, e.g.\ $\{-1, 1 \}$. Figure \ref{fig: hist} shows the results from two different simulations of this model. An interesting phenomenon is observed with the starting configuration of $\{-1, 1 \}$: the values in the urn either become almost all positive or almost all negative. Furthermore, the curve of the histogram appears exponential, reflected on either side of the vertical axis. That is, if $\mu_n$ is the empirical measure defined by the labels of the $n$ balls in the urn, then $\mu_n$ converges to some scaled exponential. Another observation from simulations is that the mean of the distribution $\mu_n$ varies from with different trials, i.e.\ there is not a deterministic limit.

The appropriately rescaled mean of the empirical measure defined by the urn process converges almost surely to some random variable. The distribution of the limiting random variable is dependent on the initial configuration of the urn. The distribution of a random draw from the urn then converges to an exponential. This is our main result Theorem \ref{thm: introTheorem}, described in the following section.

\begin{center}
\begin{figure}[t]
\begin{center}
  \begin{subfigure}[b]{0.45\textwidth}
    \includegraphics[scale = 0.45]{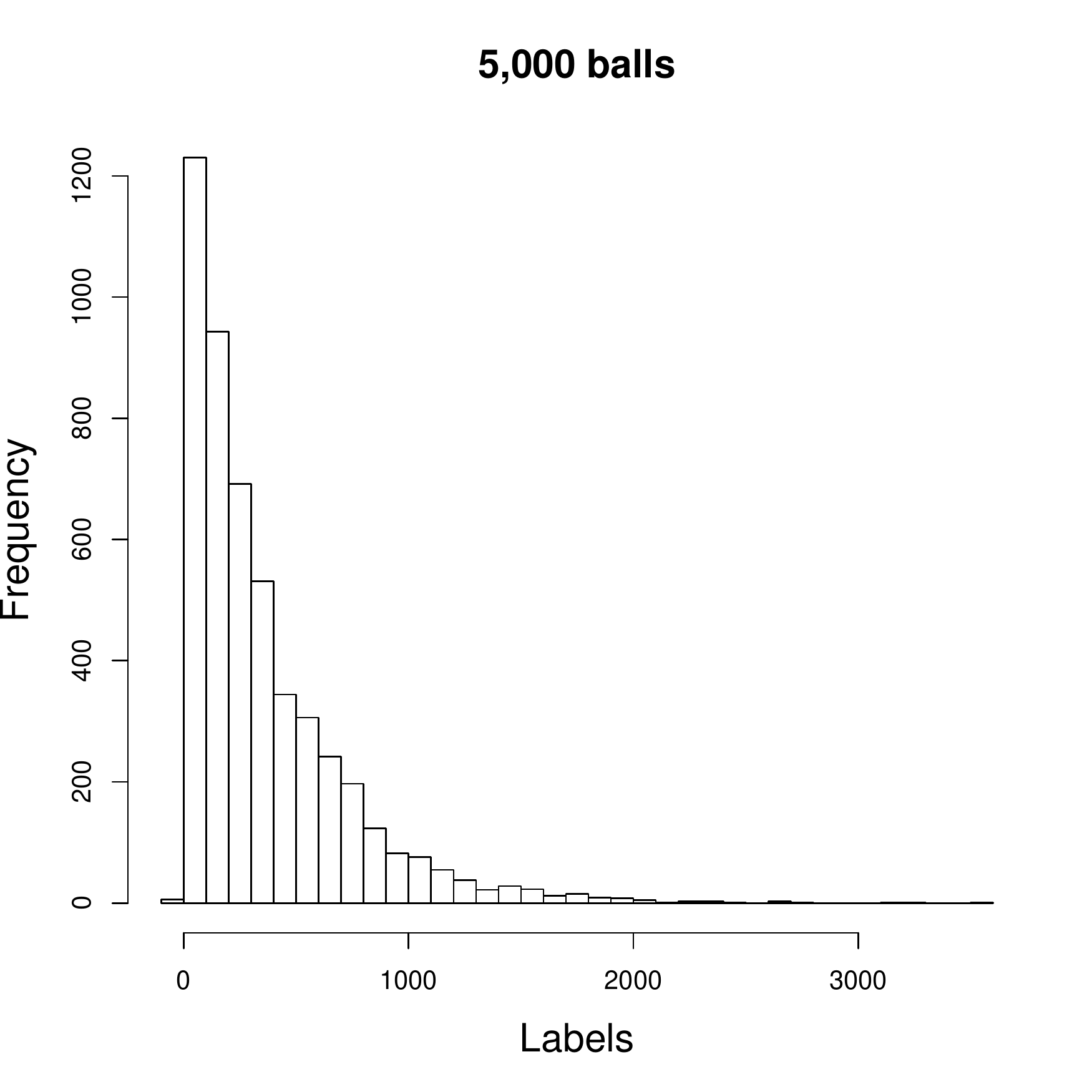}
  
  \end{subfigure}
  \begin{subfigure}[b]{0.45\textwidth}
    \includegraphics[scale = 0.45]{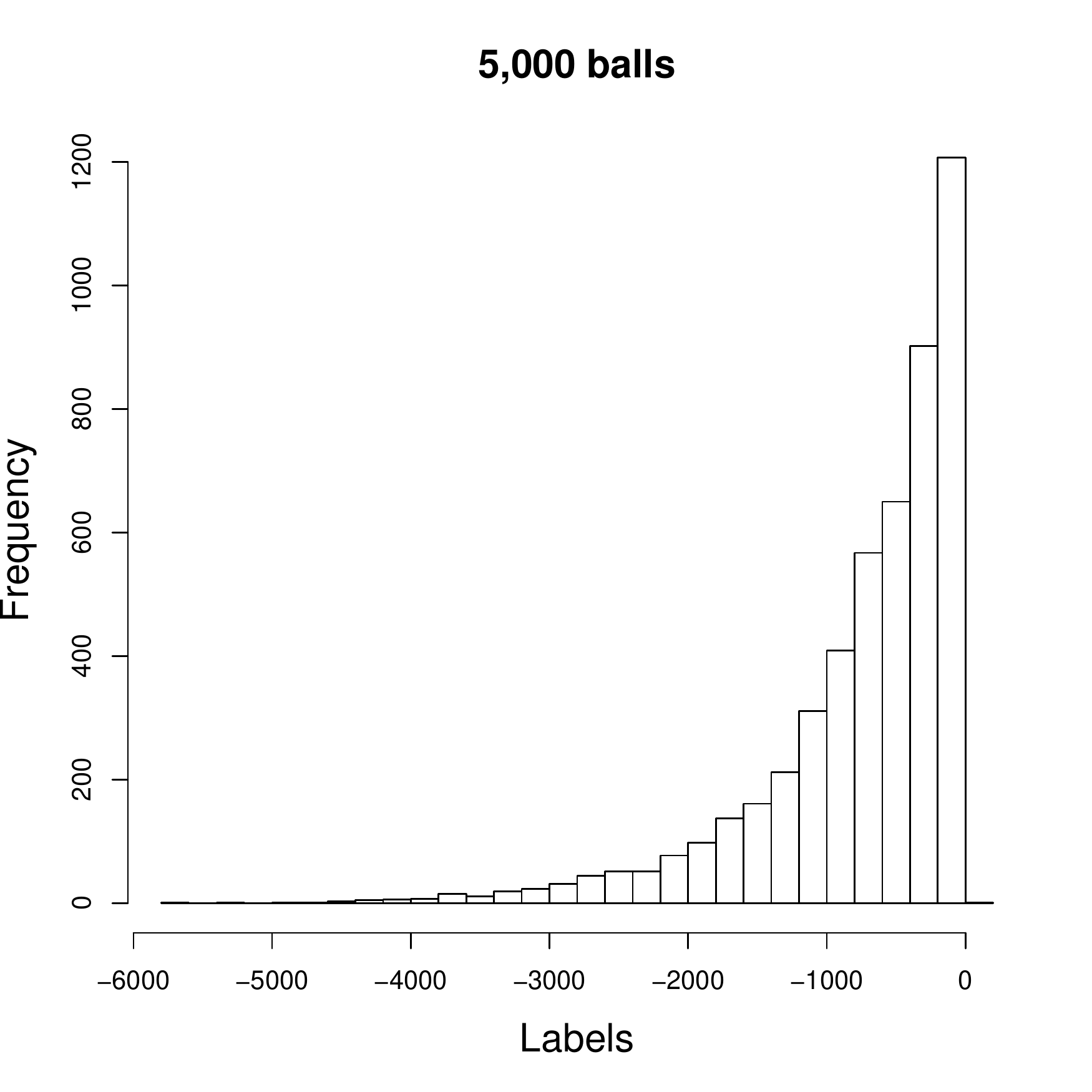}

  \end{subfigure}
  \end{center}
  
  \caption{Histograms of two different trials of the urn model with initial configuration $\{-1, 1 \}$, after 5,000 balls have been added. Each histogram has 30 bins. The bin width of the left histogram is $100$, with mean $356$, and the width of the right histogram is $200$, with mean $-737$. }
  \label{fig: hist}
\end{figure}

\end{center}

\subsection{Model and results}

Let $\tau_0$ be the number of balls initially started in the urn. For convenience, we begin the time index at $\tau_0 + 1$, so that at time $n > \tau_0$ there are exactly $n$ balls in the urn (and the number of additions is $n - \tau_0$). Let $X_n \in \Z^d$ denote the label of the $n$th ball in the urn. When $d > 1$, write $X_n(i)$ for the $i$th coordinate. Thus, the urn at time $n$ is given by the set
\[
\mbU_n = \mbU_0 \cup \{X_{\tau_0 + 1}, \dots, X_n \}.
\]
The letter $Z_n$ is used to denote a random draw from the urn, that is $Z_n \sim \mu_n$, where $\mu_n$ is the empirical distribution
\[
\mu_n = \frac{1}{n} \left( \sum_{x \in \mbU_0} \delta_x + \sum_{i = \tau_0 + 1}^n \delta_{X_i} \right), \,\,\,\,\, n \ge \tau_0.
\]
Then $X_{n+1}$ is generated by
\[
X_{n+1} = Z_n^1 + Z_n^2,
\]
where $Z_n^1 \stackrel{d}{=} Z_n^2 \stackrel{d}{=} Z_n$ are two independent draws from the urn at step $n$. For convenience, assign an arbitrary indexing to the initial balls $x \in \mathbf{U}_0$, writing them $X_1, X_2, \dots, X_{\tau_0}$.  Let $S_n = \sum_{i = 1}^n X_i$ be the sum of the labels of balls in the urn at time $n$.  When studying $Z_n$ and $X_n$ we focus on the quenched values. That is, there is an implicit underlying urn processes, and expectations are calculated conditional on the appropriate step of the urn. For instance, the mean of $Z_n$ is
\[
\E[Z_n \mid \mu_n] = \frac{S_n}{n}.
\]

\begin{theorem} \label{thm: introTheorem}
Let $\{\mu_n \}_{n \ge 0}$ be a sequence of empirical measures defined by the $\Z$-urn with any initial configuration. Suppose $Z_n \sim \mu_n$ is a random draw from the urn. Let $A = \lim_{n \to \infty} \frac{\E[Z_n \mid \mu_n]}{n+1}$. Then $A$ exists almost surely, $\P(A \neq 0) > 0$, and
\[
\lim_{n \to \infty} \mathcal{L} \left( \frac{Z_n}{n} \mid \mu_n \right) = \mathcal{L}(Z),
\]
 where $Z \sim A \cdot \text{Exp}(1)$. 
\end{theorem}

In particular, this result implies that if the urn is started with $\{-1, 1 \}$, the limiting distribution defined by the urn will be supported either on $(0, \infty)$ or $(-\infty, 0)$, depending on the result for $A$. As remarked before, Theorem \ref{thm: introTheorem} is a statement about the \emph{quenched}
version of the problem. This means a realization of the urn process $\{\mu_n \}$ is fixed and the random variable $Z_n$ is sampled from the urn process at stage $n$. In contrast, let $Z_n^*$ denote the \emph{annealed} model. That is, $Z_n^*$ is the random variable generated by first generating a new urn $\mu_n$, and then sampling from $\mu_n$.  In other words, $Z_n^*$ is the mixture of $Z_n$ over all realizations of $\mu_n$. The limiting distribution of $\frac{Z_n^*}{n}$ is the mixture distribution of exponential multiplied by the random variable $A$. At the moment, there is not an explicit form for the distribution of $A$, though one can observe how the expected value of $A$ depends on the initial configuration of the urn.  

All results are the same for urns with labels taking values in the real numbers. The motivation for introducing the process for integer value labels comes from a) the original desire to extend the model from \cite{SiegmundYakir} to a finitely generated group and b) to somewhat simplify the possible starting configurations that can be investigated. In addition, the methods can easily be extended to urns with vector labels. The exponential limit law for $Z_n$, a \emph{draw} from the urn, follows from a limit law for $X_n$, the $n$th ball added to the urn. The key is that the limit of an appropriately scaled version of $X_n$ satisfies the distributional identity
\begin{equation} \label{eqn: keyDI}
X \stackrel{d}{=} U_1 \cdot X^1 + U_2 \cdot X^2,
\end{equation}
where $U_1, U_2$ are iid Uniform$([0, 1])$ random variables and $X^1, X^2$ are iid copies of $X$. The unique random variable that satisfies this distributional identity is a Gamma with shape $2$ and mean $\E[X]$. If $X \in \R^d$ is a vector, then it is interesting to observe that solutions to \eqref{eqn: keyDI} are of the form $G \cdot \E[X]$, where $G$ is a one-dimensional Gamma random variable. That is, each scaled coordinate of $X_n$ marginally converges to a Gamma distribution, but as a joint process every coordinate converges to a multiple of the \emph{same} Gamma random variable. The precise statement for urns with vector labels is

\begin{theorem} \label{thm: dXResult}

Let $\{\mu_n^d \}_{n \ge 0}$ be a sequence of empirical measures defined by the addition urn model on $\Z^d$. Suppose $\mX_n \in \Z^d$ is the $n$th ball added to the urn. Let $\mathbf{A} = \lim_{n \to \infty} \frac{\E[\mX_{n+1} \mid \mu_n]}{2(n+1)}$. Then $\mathbf{A} = (A(1), \dots, A(d)) \in \R^d$ exists almost surely, $\P(A(i) \neq 0) > 0$ for all $i = 1, \dots, d$, and
\[
\lim_{n \to \infty} \mathcal{L} \left( \frac{\mX_n}{n} \mid \mu_{n-1}^d \right) = \mathcal{L}(\mX),
\]
 where $\mX \sim G \cdot \mathbf{A}$, and $G \sim \text{Gamma}(2, 1)$. 
\end{theorem}

\begin{remark}
A simple generalization of the model would be: draw $k$ balls (with replacement) and add a new ball with label the sum of the $k$ drawn balls. For any integer $k \ge 2$, the same techniques used in this paper prove that the the $n$th labeled added $X_n$ has limiting distribution $X$ satisfying the identity
\[
X \stackrel{d}{=} \sum_{i = 1}^k U_i \cdot X^i,
\]
where $U_i$ are iid Uniform$([0, 1])$ and $X^i$ are iid copies of $X$, for $1 \le i \le k$. Note that when $k > 2$, a Gamma distribution does \emph{not} satisfy this distributional identity. 
\end{remark}

\subsection{Related work}

Random processes defined using urns can arise is many different settings. For a plethora of example models and applications, see the book \cite{JohnsonKotz} and the survey \cite{pemantleSurvey}. ``P\'{o}lya-type'' urn models refer to an urn containing different types of balls with some sort of replacement scheme. That is, balls are drawn randomly from the urn and replaced with some number and type of balls that depends on the type of the drawn balls \cite{Mahmoud}.

Urn models are well-studied for a finite number of colors (i.e.\ types of balls); an extended scheme allowing a continuum of colors was proposed in the classic paper by Blackwell and MacQueen \cite{BlackwellMacQueen}. More work in this setting has been done in the past few years, and a general infinite-color model was introduced in \cite{thackerThesis}. In \cite{TBInfinite17}, the model on urns containing countably infinite number of colors is proposed and when the replacement schemes are balanced and associated with bounded increment random walks on the color space (e.g.\ $\Z^d$), central and local limit theorems for the random color of the $n$-th selected ball can be proven.  This has been extended to almost sure convergence, with assumptions on the replacement scheme \cite{Janson2018}. In all of these works, the urns are updated after drawing a \emph{single} ball. The replacement rule is a function of a single variable of the color space.

Much less work has been done for models in which \emph{multiple} balls are drawn at each stage. Several specific cases have been analyzed, for example in Chapter 10 of \cite{Mahmoud}. The first general results for two-color urns were contained in \cite{KubaMahmoudMulti}, \cite{KubaMahmoudMultiII}. The results were extended to arbitrary (but finite) $d$-color urns, with an additional assumption removed, in \cite{MaillerMulti}. The method used in \cite{MaillerMulti} is \emph{stochastic approximation}, which is useful for a broad class of urn models. Though stochastic approximation processes can be defined for infinite-dimensional random processes, it is challenging to apply to the $\Z$-urn because new types are introduced as the process progresses. To the best of my knowledge, the $\Z$-urn is the first model to be analyzed that involves both multiple drawings and an infinite type of balls.  

Motivation for such an urn model comes from philosophy: a reinforcement learning process for a signaling game which models the evolution of a simple language \cite{SkyrmsBook}. The reinforcement learning can be viewed as interacting P\'{o}lya-type urn system, and each step involves multiple drawings. In \cite{Invention}, the reinforcement procedure is extended to include the feature of \emph{invention}, in which new strategies are introduced as time progress, and so the urn model contains an infinite type of balls. Mathematical analysis of the convergence of this process remains incomplete.

\subsection{Outline}

The outline of this paper is as follows. In Section \ref{sec: mean}, a martingale is used to prove that the rescaled mean converges. Though the rescaled mean is not bounded, the second moment can be bounded, thus applying Doob's $L_2$ martingale convergence theorem. In addition, recursive arguments prove that all moments of the limiting random variable exists. In Section \ref{sec: contraction}, the contraction method is outlined and a recursive distributional equation is used to prove Theorems \ref{thm: introTheorem} and \ref{thm: dXResult}. The final section contains conclusions, future questions, and acknowledgments.

\section{A martingale for the mean} \label{sec: mean}

In this section, a scaled version of the mean of the empirical distribution determined by the urn is studied. Specifically, recall that $S_n = \sum_{i = 0}^n X_i$ is the sum of the labels of the first $n$ balls in the urn, and let $A_n := \frac{S_n}{n(n+1)} \in \R^d$.

\begin{lemma} \label{lem: An}
 The process $\{A_n \}_{n \ge \tau_0}$ is a martingale and 
 \[
 0 < \liminf_{n \ge \tau_0} \E[A_n(i)^2] \le \limsup_{n \ge \tau_0} \E[A_n(i)^2] < \infty,
 \]
 for $1 \le i \le d$. Hence, $A_n$ converges almost surely and in $L_2(\R^d)$. 
\end{lemma}

\begin{proof}
Conditional on the current configuration of the urn, the expectation of the $(n + 1)$th ball added is
\[
\E[X_{n+1} \mid \mu_n] = 2 \cdot \E[Z_n \mid \mu_n] = \frac{2 S_n}{n}, \,\,\,\,\, n \ge \tau_0.
\]
Then,
\begin{align*}
\E[A_{n+1} \mid \mu_n] &= \frac{1}{(n + 1) \cdot (n + 2)} \cdot \E[S_n + X_{n+1} \mid \mu_n] \\
&= \frac{1}{(n + 1) \cdot (n + 2)} \left( \frac{(n + 2) S_n}{n} \right) = \frac{S_n}{n \cdot(n + 1)} = A_n,
\end{align*}
and thus $\{ A_n \}_{n \ge \tau_0}$ is a martingale. Now, the \emph{annealed} expectation of the sum $S_n$ is
\[
\E[S_n] = n \cdot (n+1) \E[A_n] = n \cdot (n+1) \E[A_{\tau_0}] = \frac{n(n+1)}{\tau_0(\tau_0 + 1)} \cdot S_{\tau_0},
\]
for $n > \tau_0$. Note that for a fixed $n$, each coordinate $S_n(i)$ is almost surely bounded. However, the size of the largest possible ball in the urn is growing exponentially. To prove convergence of the martingale $A_n$, one can prove that each coordinate $A_n(i)$ is $L_2$ bounded. For the following, omit writing the index and so each quantity is one-dimensional. Define $R_n := S_n^2 = \sum_{i, j = 1}^n X_i X_j$, and also define $Q_n := \sum_{i = 1}^n X_i^2$, for the purpose of writing
\[
\E \left[ X_{n+1}^2 \mid \mu_n \right] = \frac{1}{n^2} \sum_{i, j = 1}^n \left( X_i + X_j \right)^2 = \frac{2 Q_n}{n} + \frac{2 R_n}{n^2}.
\]
Using this, calculate the conditional expectation
\begin{align*}
\E \left[ R_{n+1} \mid \mu_n \right] &= \E \left[ (S_n + X_{n+1})^2 \mid \mu_n \right] \\ 
&= S_n^2 + 2 S_n \E[X_{n+1} \mid \mu_n] + \E[X_{n+1}^2 \mid \mu_n ] \\
&= R_n \left( 1 + \frac{4}{n} + \frac{2}{n^2} \right) + \frac{2 Q_n}{n}.
\end{align*}
And similarly one can compute
\begin{align*}
\E \left[ Q_{n+1} \mid \mu_n \right] &= \E \left[ Q_n + X_{n+1}^2 \mid \mu_n \right] \\ 
&= Q_n \left( 1 + \frac{2}{n}  \right) + \frac{2 R_n}{n^2}.
\end{align*}
After taking expectations, for $n > \tau_0$ we get the system of recursions
\begin{align*}
\begin{pmatrix} \E[R_{n+1}] \cr \E[Q_{n+1}] \end{pmatrix} = \begin{pmatrix} 1 + \frac{4}{n} + \frac{2}{n^2} & \frac{2}{n} \cr \frac{2}{n^2} & 1 + \frac{2}{n} \end{pmatrix} \cdot \begin{pmatrix} \E[R_{n}] \cr \E[Q_{n}] \end{pmatrix}.
\end{align*}

This recursion works to prove that $\E[R_n] \le C n^4$ and $\E[Q_n] \le 2 Cn^3$, for some constant $C$. Indeed, assume for induction the statement is true for $n$. Then,
\begin{align*}
\E[R_{n+1}] &\le \left( 1 + \frac{4}{n} + \frac{2}{n^2}\right) C n^4 + \frac{2}{n} 2Cn^3 \\
&= C \left( n^4 + 4 n^3 + 6 n^2 \right) \le C(n + 1)^4
\end{align*}
and also 
\begin{align*}
\E[Q_{n+1}] &\le \frac{2}{n^2} C n^4 + \left(1 + \frac{2}{n} \right)2C n^3 \\
&= 2C( n^3 + 3n^2) \le 2C(n + 1)^3. 
\end{align*}
The induction can be initialized for $C = R_{\tau_0}$.


Note that $A_n^2 = R_n/(n^2(n + 1)^2)$, thus $\limsup_{n \ge 0} \E[A_n^2] \le C$, and so by the $L_2$ martingale convergence theorem (e.g.\ \cite{WilliamsBook}, page 111), $A_n$ converges almost surely and in $L_2$. To prove that $\inf_{n \ge \tau_0} \E[A_n^2] > 0$, observe that 
\begin{align*}
\E[A_n^2] &\ge \prod_{k = \tau_0}^{n-1} \frac{1}{(k + 2)^2} \left( k^2 + 4k + 2 \right) \cdot A_{\tau_0} \\
&= \prod_{k = \tau_0}^{n-1} \left( 1 - \frac{2}{k^2 + 4k + 4} \right) \cdot A_{\tau_0} \\
& \ge \prod_{k = \tau_0}^{\infty} \left( 1 - \frac{2}{k^2 + 4k + 4} \right) \cdot A_{\tau_0}
\end{align*}
Since $\sum_{k = 0}^\infty 2/(k^2 + 4k + 4) < \infty$, the infinite product converges to a positive value. 
\end{proof}

The previous result states that the limit $\lim_{n \to \infty} A_n =: A$ exists almost surely for each realization of the urn process, and furthermore since $\E[A^2] > 0$ the limit is non-trivial. Note that this result holds for \emph{any} initial configuration of the urn. The exact distribution of the random variable, $A$, is undetermined.  


%

\section{Recursive distributional equations} \label{sec: contraction}

The contraction method is a tool in the probabilistic analysis of algorithms for which a recursive distributional equation exists. It was developed in \cite{RoslerQuicksort} to analyze the number of comparisons required by the sorting algorithm Quicksort. A more generalized development was presented in \cite{RRcontractionIntro}. More limit theorems, along with numerous applications to recursive algorithms and random trees, were discussed in \cite{NRgeneralLimitTheorem}. The contraction method has been used to prove convergence in distribution of general P\'{o}lya-type urns \cite{PolyaContraction} with a finite number of colors,  using recursive distributions for the number of balls of each color, started from some initial configuration. 

The idea of the contraction method is to find a fixed point equation from the recursive distributional equations. Just as in a deterministic recursion, one expects the sequence to converge to the fixed point of the transformation. However, in the setting of random variables care must be taken in choosing a metric in which the recursive equation is a contraction. The correct metric is introduced in Section \ref{sec: metrics}, which also contains the necessary distributional identities. The main contraction method proof is in Section \ref{sec: mainProof}, and it is applied in Section \ref{sec: finalProof} with the correct scaling for the $\Z$-urn.

\subsection{Set-up} \label{sec: metrics}

Let $\M^{\R^d}$ denote the space of all probability distributions on $\R^d$ with the Borel $\sigma$-field. For a vector $\mx \in \R^d$, let $\|\mx \|$ denote the usual Euclidean norm. Let $\mu \in \R^d$, and define the sub-spaces:
\begin{align*}
 \M^{\R^d}_s &:= \left\lbrace \mathcal{L}(X) \in \M^{\R^d} : \E[\|X \|^s ] < \infty \right\rbrace, \\
 \M^{\R^d}_s(\mu) &:= \left\lbrace \mathcal{L}(X) \in \M^{\R^d}_s : \E[X ] = \mu \right\rbrace.
\end{align*}
For $1 \le p < \infty$ and two probability distributions $\nu, \rho \in \M^{\R^d}_p$, the minimal $L_p$ (or Wasserstein $L_p$) metric, $\ell_p$ is defined
\begin{equation} \label{eqn: lpDef}
\ell_p(\nu, \rho) := \inf \left\lbrace \|V - W \|_p : \mathcal{L}(V) = \nu, \mathcal{L}(W) = \rho \right\rbrace,
\end{equation}
where $\|V - W \|_p := ( \E[\|V - W \|^p] )^{1/p}$ is the usual $L_p$-distance. For random variables $X, Y$, we will write $\ell_p(X, Y)$ to mean $\ell_p(\mathcal{L}(X), \mathcal{L}(Y))$. The infimum in \eqref{eqn: lpDef} is obtained for all $\nu, \rho \in \M^\R_1$, and the random variables $V \sim \nu, W \sim \rho$ which achieve the infimum are called the \emph{optimal coupling}.   For proofs of this fact and the following lemma, see Section 8 in \cite{bootstrap}. 

\begin{lemma} \label{lem: l2}
Suppose $\{X_n \}$ is a sequence of random variables in $\M^{\R^d}_p$ and $X \in \M^{\R^d}_p$. Then $\ell_p(X_n, X) \to 0$ if and only if $X_n \to X$ weakly and $\E[\|X_n\|^p] \to \E[\|X \|^p]$.
\end{lemma}

The contraction method gives a distributional identity that the limiting random variable must satisfy. For the purpose of characterizing the limiting random variable, note the following distributional identity that exponential distributions satisfy. This can easily be calculated, or see \cite{exponentialFixedpt}.

\begin{lemma} \label{lem: expFixedPoint}
Let $Y$ be a random variable satisfying $Y \ge 0$ with probability 1 and
\[
Y \stackrel{d}{=} U \cdot (Y^{1} + Y^{2}),
\]
where $U$ is uniformly distributed on $[0, 1]$ and $Y^{1} \stackrel{d}{=} Y^2 \stackrel{d}{=} Y$, and the three random variables $U, Y^1, Y^2$ are independent. Then $Y$ has an exponential distribution.
\end{lemma}

In addition, we need the following characterization of a random vector in which each coordinate is a multiple of the same one-dimensional gamma random variable. 

\begin{lemma} \label{lem: gamma}
Let $\mX \in \R^d$ be a random vector satisfying $\mX(i) \ge 0$ with probability $1$ for all $1 \le i \le d$ and
\[
\mX  \stackrel{d}{=} U_1 \cdot \mX^1 + U_2 \cdot \mX^2,
\]
where $U_1, U_2$ are Uniform$([0, 1])$ random  variables and $\mX^1 \stackrel{d}{=} \mX^2  \stackrel{d}{=} \mX$, and $U_1, U_2, \mX^1, \mX^2$ are independent. Suppose $\mathbf{m} = \E[\mX] \in \R^d$. Then $\mX \sim G \cdot \mathbf{m}$, where $G \sim Gamma(2, 1)$. 
\end{lemma}

\begin{proof}
To determine the distribution of $\mX$, one can calculate the joint characteristic function:  For $\mathbf{t} = (t_1, \dots, t_d) \in \R^d$, the result is
\begin{align*}
\varphi_{\mX}(\mathbf{t}) &= \E \left[ \exp \left( i\sum_{j = 1}^d t_j \mX(j) \right)\right] = \E \left[ \exp \left( i\sum_{j = 1}^d t_j \left( U_1\mX(j)^1 + U_2 \mX(j)^2 \right) \right)\right] \\
&= \E \left[ \exp \left( i U_1 \sum_{j = 1}^d t_j \mX(j)^1 \right) \right] \cdot  \E \left[ \exp \left( i U_2 \sum_{j = 1}^d t_j \mX(j)^2 \right) \right] \\
&= \left( \int_0^1  \E \left[ \exp \left( i u \sum_{j = 1}^d t_j \mX(j) \right) \right] \, du \right)^2 = \left( \int_0^1 \varphi_{\mX}( u \cdot \mathbf{t}) \, du \right)^2
\end{align*}
The function $\varphi_{\mX}$ also satisfies
\[
\begin{cases}
\varphi_{\mx}(\mathbf{0}) = 1 \\
\frac{\partial}{ \partial t_j} \varphi_{\mx}(\mathbf{0}) = \E[i \cdot \mX(j)] = 2i \cdot \mathbf{A}(j) \,\,\,\,\, 1\le j \le d
\end{cases}.
\] 
The unique solution to this integral equation is
\begin{align} \label{eqn: characteristic}
\varphi_\mX(\mathbf{t}) = \left( 1 - i \cdot \sum_{j = 1}^d t_j \mathbf{A}(j) \right)^{-2}.
\end{align}
Equation \eqref{eqn: characteristic} is exactly the characteristic function for the random vector $G \cdot \mathbf{A}$, where $G$ is a $\text{Gamma}(2, 1)$ random variable. 
\end{proof}

\subsection{Contraction method result} \label{sec: mainProof}

To write down the defining recursive distributional equation for the $\Z$-urn, recall that the time index is started at $\tau_0$. If the initial configuration is $\mbU = \{x_1, \dots, x_{\tau_0} \}$, define $X_1 = x_1, X_2 = x_2,\dots, X_{\tau_0} = x_{\tau_0}$. The newly added ball $X_n$ is the sum of two randomly chosen balls that were added in the past. This gives the distributional identity

\begin{equation} \label{eqn: xnDist}
X_n \stackrel{d}{=} X_{I_n^1}^1 + X_{I_n^2}^2, \,\,\,\, n > \tau_0
\end{equation}
where for $i =  1, 2$, $\{X_n^i \}$ are iid copies of $\{X_n \}$ and $I_n^i$ are indices drawn uniformly from $\{1, \dots, n - 1 \}$. To prove convergence, it is necessary to scale $X_n$ appropriately.  Define $\tilde{X}_n = \frac{X_n}{n}$. Then Equation \eqref{eqn: xnDist} becomes
\begin{equation} \label{eqn: xnTildeDist}
\tilde{X}_n \stackrel{d}{=} B_{I_n^1} \cdot \tilde{X}_{I_n^1}^1 + B_{I_n^2} \cdot \tilde{X}_{I_n^2}^2,
\end{equation}
where $B_{I_n^i} = \frac{I_n^i}{n}$.  The correct fixed point equation from this recursion arises from the fact that $\frac{I_n^i}{n}$ converges to a $\text{Uniform}([0, 1])$ random variable in the $\ell_2$ metric. This can be proven by constructing the variables as follows: Let $U$ be $\text{Uniform}([0, 1])$ and define
\[
I_n = \sum_{i = 0}^{n-1} i \cdot \textbf{1} \left\lbrace \frac{i}{n} \le U < \frac{i+1}{n} \right\rbrace.
\]
Then $I_n \sim \text{Uniform}(\{1, \dots, n - 1\})$. Clearly $\frac{I_n}{n} \to U$ weakly and since second moment converges to $1/3$, the convergence is in $\ell_2$ by Lemma \ref{lem: l2}. The following is a general result about random variables satisfying the distributional recursion Equation \eqref{eqn: xnTildeDist}. The result is a special case of Theorem 4.1 and Corollary 4.2 from \cite{NeiningerMultiVariate}, though since the proof is short it is included for completeness.

\begin{theorem} \label{thm: tildeXnConverge}
Let $\{\tilde{X}_n \}_{n \ge 1}$ be a sequence with $\tilde{X}_n \in \M_2^{\R^d}$ and $\lim_{n \to \infty} \E[\tilde{X}_n] = \mu$. Let $\{B_n \}_{n \ge 1}$ be a sequence with $B_{n} \in \M_2^\R$ and $\E[B_n] \to 1/2$, and if $I_n \sim \text{Uniform}(\{1, \dots, n-1 \})$, then the sequence $\{B_{I_n} \}_{n \ge 1}$ converges in $\ell_2$ to a $\text{Uniform}([0, 1])$ variable.

For $i = 1, 2$, let $I_n^i$ be independent $\text{Uniform}(\{1, \dots, n-1 \})$ random variables, $\{B_n^i \}$ independent copies of $\{B_n \}$, and $\{\tilde{X}_n^i \}$ independent copies of $\{X_n \}$.  Suppose that for all $n \ge \tau_0$,
\begin{equation} \label{eqn: xnTildeDistInThm}
\tilde{X}_n \stackrel{d}{=} B_{I_n^1}^1 \cdot \tilde{X}_{I_n^1}^1 + B_{I_n^2}^2 \cdot \tilde{X}_{I_n^2}^2.
\end{equation}

 Then $\tilde{X}_n$ converges in distribution to a random variable $X \in \M_2^{\R^d}(\mu)$ which satisfies
\begin{equation} \label{eqn: gamDist}
X \stackrel{d}{=} U_1 \cdot X^1 + U_2 \cdot X^2,
\end{equation}
where $U_1, U_2$ are $\text{Uniform}([0, 1])$ random variables, and $X \stackrel{d}{=} X^1 \stackrel{d}{=} X^2$, and $U_1, U_2, X^1, X^2$ are independent.
\end{theorem}

\begin{proof}
The result is proven for $d = 1$, the multi-dimensional result being analogous. Let $X$ be a random variable satisfying \eqref{eqn: gamDist}. Let $(B_{I_n^1}, U_1)$ and $(B_{I_n^2}, U_2)$ be the optimal couplings, such that $\ell_2^2(B_{I_n^i}, U_i) = \E[ (B_{I_n^i} - U_i)^2] \to 0$. Define $d_n = \ell_2^2(\tilde{X}_n, X)$. Then,
\begin{align}
d_n & \le \E\left[ ( (B_{I_n^1} \tilde{X}_{I_n^1}^1 + B_{I_n^2} \tilde{X}_{I_n^2}^2) - (U_1 X^1 + U_2 X^2) )^2 \right] \notag \\ 
&=  2 \E \left[ (B_{I_n^1}  \tilde{X}_{I_n^1}^1 - U_1 X^1)^2 \right]  + 2 \E \left[ (B_{I_n^1} \cdot \tilde{X}_{I_n^1}^1 - U_1 X^1)(B_{I_n^2} \cdot \tilde{X}_{I_n^2}^2 - U_2 X^2) \right]. \label{eqn: first}
\end{align}
Observe that the assumptions imply, for $I_n \sim \text{Uniform}(\{1, \dots, n - 1 \})$, that
\[
\E[B_{I_n} \tilde{X}_{I_n}] = \sum_{j = 0}^{n-1} \frac{1}{n} \E[B_j] \E[\tilde{X}_j] \to \frac{\mu}{2}.
\]
This along with the independence of the variables involved gives that the second term in \eqref{eqn: first} converges to $0$ as $n \to \infty$. Now expanding the first term, and disregarding super-scripts:
\begin{align}\label{eqn: term}
 \E \left[ (B_{I_n}  \tilde{X}_{I_n} - U X)^2 \right] &=  \E \left[ \left( B_n (\tilde{X}_{I_n} - X) + X (B_{I_n} - U) \right)^2 \right] \notag \\ 
&=  \E \left[ (B_{I_n})^2   (\tilde{X}_{I_n} - X)^2 \right] +  \E \left[ X^2 (B_{I_n} - U)^2 \right] + 2 \E \left[ B_{I_n} X (\tilde{X}_{I_n} - X)(B_{I_n} - U) \right].
\end{align}
Note that $\E \left[ X^2 (B_{I_n} - U)^2 \right] = \E[X^2] \cdot \E \left[ (B_{I_n} - U)^2 \right] \to 0$ by construction. Also, since $\E[B_{I_n}(B_{I_n} - U)] \to 0$, the final term in \eqref{eqn: term} converges to $0$ as $n \to \infty$. Thus, the final term to analyze is
\begin{align*}
\E \left[ (B_{I_n})^2 \cdot  (\tilde{X}_{I_n} - X)^2 \right] &= \sum_{j = 1}^{n-1} \E \left[(B_{I_n})^2 \cdot  (\tilde{X}_{I_n} - X)^2 \mid I_n = j \right] \cdot \P(I_n = j) \\
&= \sum_{j = 1}^{n-1} \E \left[ \frac{j^2}{n^2} (\tilde{X}_j - X)^2 \right] \cdot \frac{1}{n-1} = \sum_{j = 1}^{n-1} \frac{j^2}{n^2(n-1)} \cdot d_j.
\end{align*}
This implies,
\begin{align*}
\limsup_{n \to \infty} d_n & \le 2 \limsup_{n \to \infty} \sum_{j = 1}^{n-1} \frac{j^2}{n^2(n-1)} \cdot d_j \le \frac{2}{3} \cdot (\limsup_{n \to \infty} d_n)
\end{align*}
 This gives a contradiction unless $\limsup_{n \to \infty} d_n = 0$. Since $d_n \ge 0$, this proves that $\lim_{n \to \infty} d_n = 0$, which is the desired convergence.

\end{proof}

\begin{remark}
Theorem \ref{thm: tildeXnConverge} can also be proven by applying Theorem 4.1 in \cite{RRcontractionIntro} with the Zolotarev metric $\zeta_2$.  For $\nu, \rho \in \M^\R$, the Zolotarev distance $\zeta_s$, $s \ge 0$ is defined
\[
\zeta_s(X, Y) := \zeta_s(\mcL(X), \mcL(Y)) := \sup_{f \in \F_s} | \E[ f(X) - f(Y)] |,
\]
where $s = m + \alpha$ with $0 < \alpha \le 1,  m \in \mathbb{N}_0$ and
\[
\F_s := \{f \in C^m(\R, \R): |f^{(m)}(x) - f^{(m)}(y) | \le |x - y|^\alpha \}.
\]
The use of Zolotarev metrics in contraction methods was developed in \cite{RRcontractionIntro} because
 in some settings the recursive distributional equations do not give a strict contraction in $\ell_p$. Most important of these settings is fixed-point equations that occur for normal distributions, which can be handled by $\zeta_s$ for $s > 2$.

\end{remark}

\subsection{Limit laws} \label{sec: finalProof}

Now, Theorem \ref{thm: tildeXnConverge} can be applied to prove the result about the $\Z$-urn. The one technicality is that Theorem \ref{thm: tildeXnConverge} requires that the random variables in question have the same mean. If $X_n$ was just defined as the $n$th ball added to the urn, then this would not be true. Thus, it is necessary to re-scale by the mean of $X_n$, and use the results from Section \ref{sec: mean} that the mean converges. It is important to remember the quenched setting; all variables are conditional on the \emph{same} underlying urn process.

\begin{proof}[Proof of Theorems \ref{thm: introTheorem} and \ref{thm: dXResult}]
We again restrict to a one-dimensional process. Let $\{ \mu_n \}_{n \ge \tau_0}$ be a realization of the  $\Z$-urn process. For $n > \tau_0$, let $X_n$ denote the $n$th ball added to the urn, conditional on the urn $\mu_{n-1}$. Recall $\E[X_n \mid \mu_{n-1}]  = 2 n \cdot A_{n-1}$, where $A_n := \frac{S_n}{n(n+1)}$. Thus, define $\tilde{X}_n := \frac{X_n}{n}$, then $\lim_{n \to \infty} \E[X_n \mid \mu_{n-1}] = 2 \lim_{n \to \infty} A_n = 2 A$ exists by Lemma\ref{lem: An}.  The recursive distributional equation satisfied by $\tilde{X}_n$ is
\[
\tilde{X}_n \stackrel{d}{=} \frac{I_n^1}{n}  \cdot \tilde{X}_{I_n^1}^1 + \frac{I_n^2}{n} \cdot \tilde{X}_{I_n^2}^2,
\]
for $I_n^i, i = 1, 2$ iid $\text{Uniform}(\{1, \dots, n -1 \})$. As remarked before Theorem \ref{thm: tildeXnConverge}, the random variables $\frac{I_n^i}{n}$ converge in $\ell_2$ to the $\text{Uniform}([0, 1])$ distribution. Thus, taking $B_{I_n}^i = I_n^i/n$,
Theorem \ref{thm: tildeXnConverge} can then be applied to get that $\tilde{X}_n$ converges in distribution to some $X$ with $\E[X] = 2 \cdot A$ which satisfies \eqref{eqn: gamDist}. Finally, Lemma \ref{lem: gamma} says that the distribution of $X/A$ must be Gamma with shape $2$. Now, recalling that $Z_n$ denotes a draw from the urn $\mu_n^d$, we have the distributional identity
\begin{equation} \label{eqn: ZXrel}
Z_n \stackrel{d}{=} X_{I_n},
\end{equation}
where $I_n$ is a uniformly chosen index from $\{1, \dots, n \}$. Again scaling $\tilde{Z}_n = \frac{Z_n}{n}$, we get
\[
\tilde{Z}_n \stackrel{d}{=} \frac{I_n}{n} \cdot \tilde{X}_{I_n}.
\]
Since $\tilde{X}_{I_n}/A$ converges in $\ell_2$ to $X \sim \text{Gamma}(2, 1)$, this prove $\tilde{Z}_n$ converges in distribution to $U \cdot X$, where $U \sim \text{Uniform}([0, 1])$. By Lemma \ref{lem: expFixedPoint}, it follows that the limit point has exponential distribution.

\end{proof}

\section{Conclusion and future questions}

\begin{center}
\begin{figure}[t]
\begin{center}
  \begin{subfigure}[b]{0.45\textwidth}
    \includegraphics[scale = 0.4]{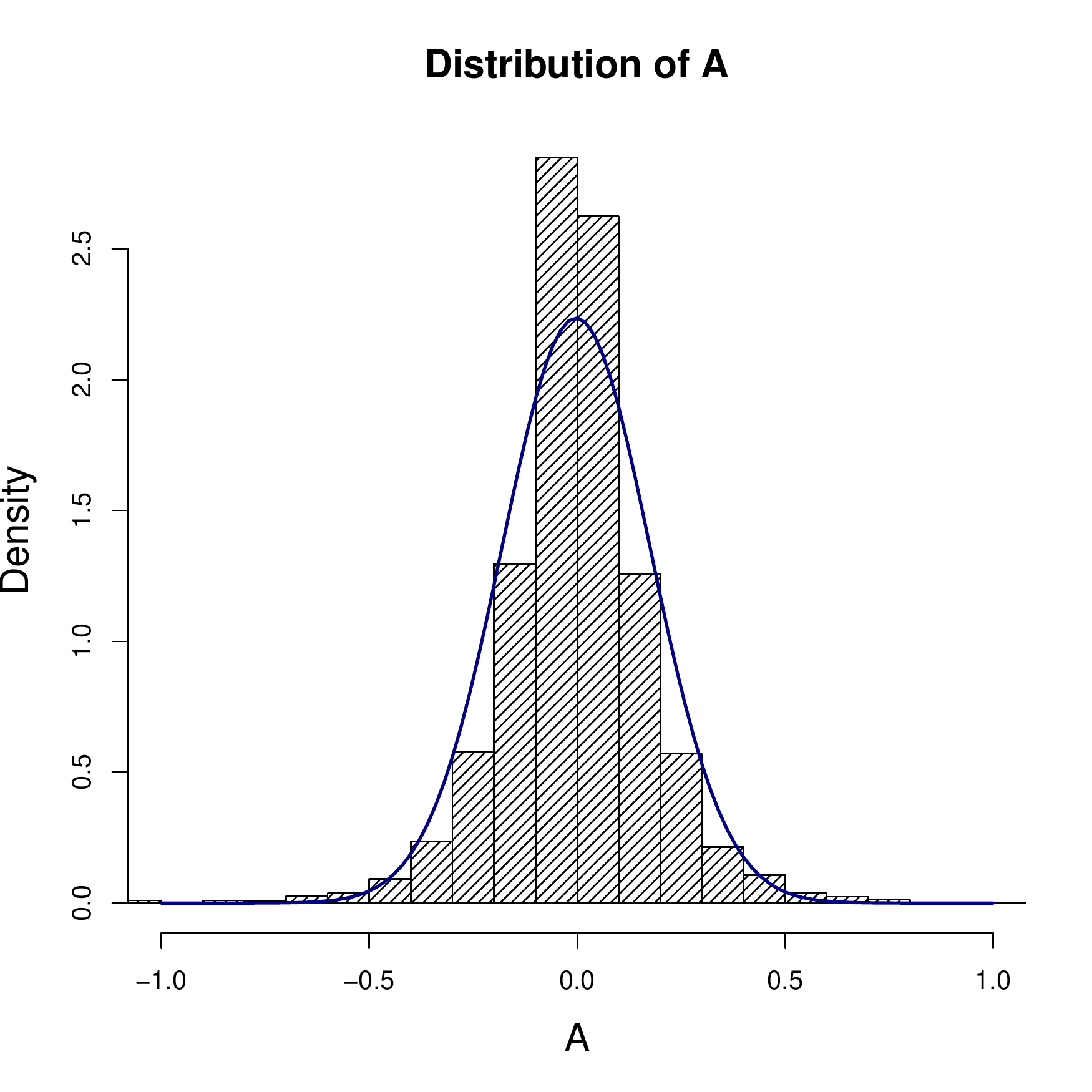}
    \caption{Initial configuration of $\{-1, 1 \}$. \\The curve is a fitted normal density.}
    \label{fig:1}
  \end{subfigure}
  \begin{subfigure}[b]{0.45\textwidth}
    \includegraphics[scale = 0.4]{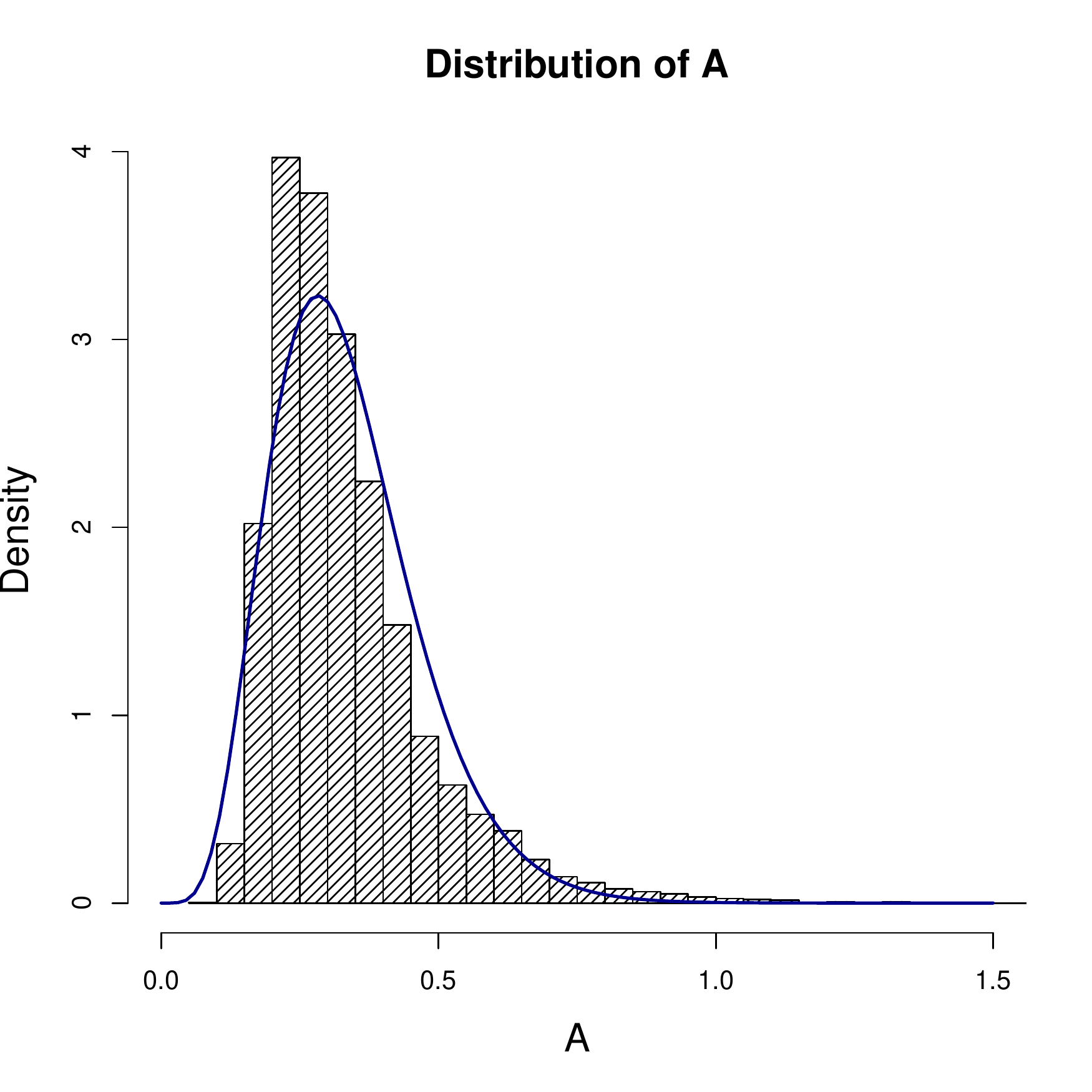}
    \caption{Initial configuration of $\{1, 1 \}$. \\The curve is a fitted Gamma density. }
    \label{fig:2}
  \end{subfigure}
  \end{center}
  
  \caption{Results for $A$ from simulations of the $\Z$-urn model from two different initial configurations. An urn realization was run for $5,000$ rounds and the value of $A_n$ was recorded. A total of $5,000$ realizations of each urn process were run to display the empirical distribution of $A$. }
  \label{fig: A}
\end{figure}

\end{center}

This paper introduced a new urn model, extending the product urn model from \cite{SiegmundYakir} for finite groups to the integers or real numbers. Among the vast collection of urn models, the $\Z$-urn is interesting because it contains two properties which are challenging to study: drawing more than one ball, and an infinite number of types of balls. It is proven that the re-scaled empirical distribution defined by the urn converges to a multiple of an exponential distribution. The multiple, determined by the mean of the limiting urn, is a random variable and a main remaining question is the distribution of this random variable.

Recall the notation $S_n = \sum_{i = 1}^n X_n$ for the sum of the labels in the urn at step $n$, and $A_n = \frac{S_n}{n(n+1)}$. The process $\{A_n \}$ is a martingale and converges almost surely to some limit $A$. The distribution of $A$ can be explored using simulations. From Section \ref{sec: mean}, the mean of $A$ is
\[
\E[A] = \frac{S_{\tau_0}}{\tau_0(\tau_0 + 1)}.
\]
That is, $A$ depends on both the sum of the initial labels in the urn, and the number $\tau_0$ of  balls started. Figure 2(a) shows a histogram for the result for $A$ for a urn started with a ball labeled $-1$ and a ball labeled $1$. The distribution is of course centered around $0$, but it is not normal. Figure 2(b) exhibits the results for $A$ for urns started with two $1$ labels. The mean in this case is $2/(2 \cdot 3) = 1/3$. A $\chi^2$ test for the null hypothesis that a Gamma distribution fits the data resulted in a $p$-value $<0.001$, indicating Gamma is not a good fit. It would be interesting to explore this distribution in future work, though even proving it is continuous seems to be a challenge. 

Another interesting problem would be to study the model for different infinite groups. For example, $\R$ with multiplication operation, or the Heisenberg group. The result of an exponential limit law is very specific to the addition process, so the limits for different groups and group operations could be curious.

\vspace{5mm}

\paragraph{Acknowledgments}  The author thanks Persi Diaconis for many discussions, Brian Skryms for introducing motivation to study urn models, and David Siegmund, Svante Janson, and Robin Pemantle for helpful comments. She also thanks an anonymous referee for careful reading and suggestions for interesting generalizations. The author is supported by a National Defense Science \& Engineering Graduate Fellowship.

\bibliographystyle{abbrv}
\bibliography{citations}

\end{document}